\documentclass[10pt,reqno]{amsart}
\usepackage[T2A]{fontenc}
\usepackage[utf8]{inputenc}
\usepackage[russian, english]{babel}
\usepackage{amsmath}
\usepackage{amssymb}
\usepackage{amsfonts}
\usepackage{bm} 
\usepackage{stmaryrd} 
\usepackage{multirow}\usepackage{lscape}
\usepackage{calrsfs} 
\usepackage{mathrsfs}
\usepackage{cite}
\allowdisplaybreaks

\newtheorem{thm}{Theorem}
\newtheorem{prop}[thm]{Proposition}
\newtheorem{lem}[thm]{Lemma}
\newtheorem{cor}[thm]{Corollary}

\newtheorem{conj}{Conjecture}

\begin{document}

\title[On generations by conjugate elements]{On  generations by conjugate elements\\ in almost simple groups with socle $\mbox{}^2F_4(q^2)'$}
\author{Danila O. Revin}%
\address{Danila O. Revin
\newline\indent Sobolev Institute of Mathematics,
\newline\indent 4, Koptyug av.
\newline\indent 630090, Novosibirsk, Russia
} \email{revin@math.nsc.ru}

\author{Andrei V. Zavarnitsine}%
\address{Andrei V. Zavarnitsine
\newline\indent Sobolev Institute of Mathematics,
\newline\indent 4, Koptyug av.
\newline\indent 630090, Novosibirsk, Russia
} \email{zav@math.nsc.ru}

\maketitle

{\small
\begin{quote}
\noindent{{\sc Abstract.} We prove that if $L=\mbox{}^2F_4(2^{2n+1})'$ and $x$~is a nonidentity automorphism of $L$ then $G=\langle L,x\rangle$ has four elements
conjugate to $x$ that generate $G$. This result is used to study
the following conjecture about the $\pi$-radical of a finite group:
Let $\pi$~be a proper subset of the set of all primes and let
$r$~be the least prime not belonging to $\pi$. Set $m=r$ if $r=2$
or $3$ and set $m=r-1$ if $r\geqslant 5$. 
Supposedly, an element
$x$ of a finite group $G$ is contained in the $\pi$-radical
$\operatorname{O}_\pi(G)$ if and only if every $m$ conjugates of
$x$ generate a $\pi$-subgroup. 
Based on the results of this paper and a few previous ones, the conjecture is confirmed for all finite groups whose every nonabelian composition factor is isomorphic to a sporadic, alternating, linear, or unitary simple group, or to one of the groups of type ${}^2B_2(2^{2n+1})$, ${}^2G_2(3^{2n+1})$,
${}^2F_4(2^{2n+1})'$, $G_2(q)$, or ${}^3D_4(q)$.}


\end{quote}
}
\section{Introduction}

In~\cite{03GuSaxl}, R.\,Guralnick and  J.\,Saxl introduced the following concept. Let $L$~be a finite simple nonabelian group and let $x\in\operatorname{Aut}(L)$~be a nonidentity (possibly, inner) automorphism. The least number $m$ such that some $m$ conjugates of $x$ in  $G=\langle\operatorname{Inn}(L),x\rangle$ generate~$G$ is denoted by   $\alpha(x,L)$. In the same paper, upper bounds on $\alpha(x,L)$ were found for every finite simple group $L$ and every automorphism $x$ of $L$ of prime order (clearly,  $\alpha(x,L)\leqslant \alpha(y,L)$ if  $y\in\langle x \rangle$~has prime order). These bounds have numerous applications in the theory of finite groups, for example, in proving generalizations and analogs of the famous Baer--Suzuki theorem~\cite{FGG,GGKP,GGKP1,GGKP2,10Guest,12Guest,13GKP,14Guest,15GuTiep,GuMa,Re,11Re,11Re1,21YangReVd,22YangWuRe,23YangWuReVd}. 

The bounds obtained in~\cite{03GuSaxl} are best possible not for all $L$. For example, more precise bounds on $\alpha(x,L)$ for sporadic groups were found in ~\cite[Theorem~3.1]{14DiMarPellZal}. 
The authors of~\cite{03GuSaxl} believe that the bounds for exceptional groups of Lie type  might not be precise either. 
It was proven~\cite[Theorem~5.1]{03GuSaxl} that if $L$~is an exceptional group of Lie type of untwisted Lie rank $\ell$ and $x$~is its automorphism of prime order then $\alpha(x,L)\leqslant \ell+3$, except in the case $L=F_4(q)$ and~$x$ being an involution, where it was proven that $\alpha(x,L)\leqslant 8$. At the same time, the authors consider it plausible   ~\cite[Remark at the end of Section~5]{03GuSaxl} that $\alpha(x,L)\leqslant 5$ for all exceptional groups of Lie type.

In this paper, we prove the following theorem:
\begin{thm}\label{thm:2F4}
Let  $L=\mbox{}^2F_4(q^2)'$, where $q^2=2^{2n+1}$. Then $\alpha(x,L)\leqslant 4$ for every  $x\in\operatorname{Aut}(L)$ of prime order.
\end{thm}
Theorem~\ref{thm:2F4} confirms and even strengthens the Guralnick--Saxl conjecture in the present case. Recall that  \cite[Theorem~5.1]{03GuSaxl} gives in this case the bound $\alpha(x,L)\leqslant 7$.  

We will also use the obtained result to study the conjecture about the sharp Baer--Suzuki theorem for the $\pi$-radical of a finite group. 

Henceforth, $\pi$ will denote a fixed subset of the set $\mathbb{P}$ of all primes such that  $\pi'=\mathbb{P}\setminus\pi\ne\varnothing$. Recall that a \emph{$\pi$-group}~is a finite group whose order has all prime divisors belonging to ~$\pi$. We set  
\begin{equation}\label{r,m}
   r=r(\pi)=\min\pi'\quad \text{and}\quad m=m(\pi)=\left\{\begin{array}{rl}
                               r, & \text{ if } r\in\{2,3\}, \\
                               r-1, &   \text{ if } r\geqslant 5.
                             \end{array}\right. 
\end{equation}
\begin{conj}\label{ConjBS} {\rm  \cite[Conjecture~1]{21YangReVd}}
    In every finite group $G$, the $\pi$-radical $\operatorname{O}_\pi(G)$, i.\,e. the largest normal $\pi$-subgroup, coincides with the set 
$$
\{x\in G\mid \langle x^{g_1},\dots,x^{g_m}\rangle\ \text{is a}\  \pi\text{-group for all}\ g_1,\dots,g_m\in G\},
$$
where $\pi$ and $m=m(\pi)$ are as defined above.
\end{conj}
It is not possible to reduce $m$ in this conjecture \cite[Proposition~1.2]{21YangReVd}. The original Baer--Suzuki theorem states that 
$$ 
\operatorname{O}_p(G)=
\{x\in G\mid \langle x^{g_1},x^{g_2}\rangle\ \text{is a}\  p\text{-group for all}\ g_1,g_2\in G\}
$$ 
in every finite group $G$ for every prime~$p$. Therefore, Conjecture~\ref{ConjBS} holds if $\pi=\{p\}$. It also holds if  $r(\pi)=2$ \cite[Theorem~1]{11Re}. 

Conjecture~\ref{ConjBS} was studied in  \cite{21YangReVd,22YangWuRe,23YangWuReVd,Re} and was confirmed for every finite group $G$ whose every nonabelian composition factor is isomorphic to a sporadic, alternating, linear, unitary group, or one of the exceptional groups ${}^2B_2(2^{2n+1})$, ${}^2G_2(3^{2n+1})$, $G_2(q)$, or ${}^3D_4(q)$. The results presented in this paper allow us to include in this list the series of exceptional groups ${}^2F_4(2^{2n+1})$ and the Tits group ${}^2F_4(2)'$. This is enabled by Theorem~\ref{thm:2F4beta} stated below. 

In order to formulate Theorem~\ref{thm:2F4beta}, we denote by $\beta_s(x,L)$, analogously to $\alpha(x,L)$, for a finite simple group $L$, a prime divisor $s$ of its order, and a nonidentity $x\in\operatorname{Aut}(L)$, the least number $m$ such that some $m$ conjugates of $x$ in  $\langle\operatorname{Inn}(L),x\rangle$ generate a subgroup of order a multiple of~$s$. Clearly,  $\beta_s(x,L)\leqslant\alpha(x,L)$.

\begin{thm}\label{thm:2F4beta}
Let $L=\mbox{}^2F_4(q^2)'$, where  $q^2=2^{2n+1}$, and let $s$~be an odd prime divisor of~$|L|$.
Then $\beta_s(x,L)\leqslant 4$ and $\beta_3(x,L)\leqslant 3$ for every  automorphism $x$ of~$L$ of prime order.
\end{thm}

Clearly, the bound $\beta_s(x,L)\leqslant 4$ in Theorem~\ref{thm:2F4beta} readily follows from Theorem~\ref{thm:2F4}. The refinement of this bound for $s=3$ requires a rather subtle additional argument. Theorem~\ref{thm:2F4beta} and the results of~\cite{11Re,21YangReVd,22YangWuRe,23YangWuReVd,Re} imply

\begin{cor}\label{BaerSuzuki}
Let $\pi$~be a nonempty set of primes not containing all primes, and let $r$~be the least prime not belonging to $\pi$. Set
$$m=\left\{\begin{array}{rl}
                               r, & \text{ if } r\in\{2,3\}, \\
                               r-1, &   \text{ if } r\geqslant 5.
                             \end{array}\right.
$$
Then
$$
\mathrm{O}_\pi(G)=\{x\in G\mid \langle x^{g_1}, \dots, x^{g_m}\rangle\text{~is a~} \pi\text{-group for every~} g_1,\dots g_m\in G\}
$$
for every finite group $G$ whose every nonabelian composition factor is isomorphic to a sporadic, alternating, linear, unitary group, or to one of the groups  ${}^2B_2(2^{2n+1})$, ${}^2G_2(3^{2n+1})$, ${}^2F_4(2^{2n+1})'$, $G_2(q)$ or  ${}^3D_4(q)$. 
\end{cor}

\section{Preliminary facts}

Our group-theoretic notation is standard. To specify the structure of a group, we use the notation from~\cite{85Atlas}. The classical groups are denoted as in ~\cite{Bray}. The exceptional groups of Lie type ${}^2B_2(q^2)$, ${}^2G_2(q^2)$, and ${}^2F_4(q^2)$ are denoted following \cite{72Car}, whereas for other groups of Lie type we use the notation from~\cite{98GorLySol}. 

We denote by $n_p$ the $p$-part of $n$, i.\,e. the largest power of a prime $p$ that divides $n$. Given natural numbers $l,n>1$, a prime divisor $p$ of $l^n-1$ is {\em primitive} if it does not divide $l^i-1$ for all $i=1,\ldots,n-1$.
It follows from \cite{Bang,Zs} that a primitive prime divisor of $l^n-1$ exists unless $(n,l)=(6,2)$ or $(2,2^s-1)$ for $s\geqslant 2$. Primitive divisors of $l^n-1$ are also divisors of  $\Phi_n(l)$, where $\Phi_n$ is the $n$th cyclotomic polynomial. In the case $l=q^2=2^{2n+1}$, 
the numbers $\Phi_4(l)=l^2+1$ and $\Phi_{12}(l)=l^4-l^2+1$ admit factorization into integer coprime factors
\begin{equation}  \label{phi4_12_dec}
\begin{split}
l^2+1&= (l-\sqrt{2l}+1)(l+\sqrt{2l}+1), \\
 l^4-l^2+1&=(l^2-\sqrt{2l^3}+l-\sqrt{2l}+1)(l^2+\sqrt{2l^3}+l+\sqrt{2l}+1).
\end{split}
\end{equation}
Hence the sets $\Pi_4(l)$ and $\Pi_{12}(l)$ of primitive prime divisors of $l^4-1$ and $l^{12}-1$ split into the disjoint subsets 
$$
 \Pi_4(l)=\Pi_4^-(l)\cup \Pi_4^+(l)\text{~~and~~} \Pi_{12}(l)=\Pi_{12}^-(l)\cup \Pi_{12}^+(l)
$$ 
so that
$$
\text{the primes in~} \Pi_4^\varepsilon(l) \text{~divide~} l+\varepsilon\sqrt{2l}+1, \text{~and} 
$$
$$
\text{the primes in~} \Pi_{12}^\varepsilon(l) \text{~divide~} l^2+\varepsilon\sqrt{2l^3}+l+\varepsilon\sqrt{2l}+1 
$$ 
for $\varepsilon=\pm1$. According to~\cite[Theorem~2]{Grechkoseeva}, the set $\Pi_n^\varepsilon(l)$ for $n\in\{4,12\}$ is nonempty in all cases, except when  $(n,l,\varepsilon)\in \{(4,2,-),(4,2^3,-),(4,2^5,-),(12,2,-)\}$.

We will also make use of the known facts from ordinary representation theory which can be found in~\cite{06Isa}.
Given three conjugacy classes $K_i$, $i=1,2,3$, of a group $G$, let $(K_1,K_2,K_3)$ denote their {\em class multiplication coefficient}, i.e. the number of pairs $(x_1,x_2)$, $x_i\in K_i$, $i=1,2$,
such that $x_1x_2$ is a fixed element $x_3\in K_3$. This coefficient can be determined from the ordinary character table of the group using the formula
\begin{equation}\label{StructConstFormula}
(K_1,K_2,K_3) = \frac{|K_1||K_2|}{|G|}\sum_{\chi\in
\operatorname{Irr}(G)}\frac{\chi(x_1)\chi(x_2)\overline{\chi(x_3)}}{\chi(1)},
\end{equation}
see \cite[Problem (3.9), p.\,45]{06Isa}.

We conclude this section with three lemmas which give estimates on $\beta_3(x,S)$ for some classical groups $S$ of small rank over a field of characteristic~$2$. We will cite these lemmas in the proof of the theorem~\ref{thm:2F4beta}.

\begin{lem} \label{L_2} {\rm\cite[Lemma~1.9]{23YangWuReVd}} 
Let $S=PSL_2(2^n)$, $n>1$. Then  $\beta_3(x,S)\leqslant 2$ for every inner automorphism  $x$ of $S$ of prime order. 
\end{lem}

\begin{lem} \label{Sp4} Let $S=Sp_4(2^n)$, $n>1$. Then  $\beta_3(x,S)\leqslant 2$ for every inner automorphism  $x\in\operatorname{Aut}(S)$ of odd prime order. 
\end{lem}
\begin{proof} It is known  \cite[Table~8.14]{Bray} that $S$ has subgroups $SO_4^+(2^n)=GO_4^+(2^n)$ and $SO_4^-(2^n)=GO_4^-(2^n)$. By  \cite[Theorem~6]{Dye}, every element of $S$ is conjugate to an element of one of these subgroups, while an element of odd order is conjugate to an element of the socle $\Omega_4^+(2^n)\cong PSL_2(2^n)\times PSL_2(2^n)$ or $\Omega_4^-(2^n)\cong PSL_2(2^{2n})$. Consequently,  $x$ induces a nonidentity automorphism $\overline{x}$ on a group $H$ isomorphic to either $PSL_2(2^n)$ or $PSL_2(2^{2n})$. By Lemma~\ref{L_2}, we have $\beta_3(\overline{x},H)\leqslant 2$, whence the claim follows. 
\end{proof}

\begin{lem} \label{U3} Let $S=PSU_3(2^n)$, $n>1$. Then  $\beta_3(x,S)\leqslant 3$ for every automorphism  $x\in\operatorname{Aut}(S)$ of odd prime order. 
\end{lem}
\begin{proof} According to~\cite[Lemma~3.3]{03GuSaxl}, we have  $\alpha(x, S)\leqslant3$ and the claim follows. 
\end{proof}

\section{Properties of $\mbox{}^2F_4(q^2)$ and $\mbox{}^2F_4(2)'$}

The groups $\mbox{}^2F_4(q^2)$  discovered by R.~Ree
\cite{61Ree} in 1961 bear the name of large Ree groups or Ree groups in characteristic~$2$. Their construction and some properties are given in \cite{72Car,98GorLySol}, for example. Our notation for these groups agrees with~\cite{72Car} and differs from~\cite{98GorLySol}, where they are denoted by   $\mbox{}^2F_4(q)$. Recall that $q^2=2^{2n+1}$, and so $q$ is always irrational. Following~\cite{75Shin}, we will sometimes denote $q^2$ by~$l$ and  $\mbox{}^2F_4(q^2)$ by $\mbox{}^2F_4(l)$. In this notation, we have
\begin{equation*}
\begin{split}
|\mbox{}^2F_4(q^2)|&= q^{24}(q^2-1)(q^6+1)(q^8-1)(q^{12}+1)\\
&=l^{12}(l-1)(l^3+1)(l^4-1)(l^6+1).   
\end{split}   
\end{equation*}
The group $\mbox{}^2F_4(q^2)$ is simple unless $q^2=2$, in which case the derived subgroup $\mbox{}^2F_4(2)'$, called the Tits group, is simple and has index~$2$ in~$\mbox{}^2F_4(2)$. 
Our results concern the series $\mbox{}^2F_4(q^2)'$ and so include the Tits group, too.

The conjugacy classes of $\mbox{}^2F_4(q^2)$ are described in  K.~Shinoda's paper~\cite{75Shin}. As we deal with elements of prime order, some information about the conjugacy classes of unipotent and semisimple elements will be used. The former ones are precisely the involutions. The necessary facts about involutions are collected in Section~\ref{InvProperties}.
Observe that the difference  $\mbox{}^2F_4(2)\setminus\mbox{}^2F_4(2)'$ contains no involutions and hence no elements of prime order, see~\cite{85Atlas}. 
The required information about the maximal tori and conjugacy classes of semisimple elements is contained in Section~\ref{SemisimpleProperties}. 
The notation  $T_i$ and $t_j$ used in Lemma~\ref{maximal} below also appears in  Section~\ref{SemisimpleProperties}. 
We will need some facts about the irreducible ordinary characters of $\mbox{}^2F_4(q^2)$ partially included in the \textsf{CHEVIE} package for \textsf{Maple} \cite{Chevie}. 
The automorphism group of $\mbox{}^2F_4(q^2)$ is generated by the inner automorphisms and a field automorphism $\phi$ of order $2n+1$, where $q^2=2^{2n+1}$ as above. 
All non-inner automorphisms of prime order are conjugate to powers of~$\phi$~\cite[Proposition~4.9.1]{98GorLySol}.

\begin{lem} \label{maximal} {\rm\cite[Main Theorem]{91Mal}}
Up to conjugacy, every maximal subgroup of $L=\mbox{}^2F_4(q^2)$, $q^2=2^{2n+1}$, belongs to the following list:
\begin{itemize}
 \item[$(1)$] $P_a=[q^{22}]:\left(PSL_2(q^2)\times \mathbb{Z}_{q^2-1}\right)$, 
 \item[$(2)$] $P_b=[q^{20}]:\left(\mbox{}^2B_2(q^2)\times \mathbb{Z}_{q^2-1}\right)$, 
 \item[$(3)$] $\operatorname{N}_L(\langle t_4\rangle)=\operatorname{C}_L( t_4):2\cong SU_3(q^2):2$,  
\item[$(4)$] $\operatorname{N}_L(T_8)\cong (\mathbb{Z}_{q^2+1}\times\mathbb{Z}_{q^2+1}):GL_2(3)$, 
\item[$(5)$] $\operatorname{N}_L(T_6)\cong (\mathbb{Z}_{q^2-\sqrt{2}q+1}\times\mathbb{Z}_{q^2-\sqrt{2}q+1}):[96]$,  
\item[$(6)$] $\operatorname{N}_L(T_7)\cong (\mathbb{Z}_{q^2+\sqrt{2}q+1}\times\mathbb{Z}_{q^2+\sqrt{2}q+1}):[96]$,  
\item[$(7)$] $\operatorname{N}_L(T_{10})\cong (\mathbb{Z}_{q^4-\sqrt{2}q^3+q^2-\sqrt{2}q+1}):12$, 
\item[$(8)$] $\operatorname{N}_L(T_{11})\cong (\mathbb{Z}_{q^4+\sqrt{2}q^3+q^2+\sqrt{2}q+1}):12$, 
 \item[$(9)$]  $PGU_3(q^2):2$,  
 \item[$(10)$]  $\mbox{}^2B_2(q^2)\wr 2$,
 \item[$(11)$]  $B_2(q^2):2\cong Sp_4(q^2):2$,
\item[$(12)$]  $\mbox{}^2F_4(q_0^2)$ if $q_0^2=2^{2m+1}$ with $(2n+1)/(2m+1)\in\mathbb{P}$.
\end{itemize}
Furthermore, a field automorphism~$f$ of~$L$ leaves fixed a representative of each of the above   listed conjugacy classes of maximal subgroups and acts on this representative in an obvious way. 
\end{lem}

\section{Properties of involutions in $\mbox{}^2F_4(q^2)$}\label{InvProperties}

\begin{lem} \label{involutions} The group 
$L=\mbox{}^2F_4(q^2)$ contains precisely two conjugacy classes of involutions $2A$ and $2B$, where the notation is such that class $2A$ consists of the $2$-central elements, and class $2B$~of the non-$2$-central elements. Both classes are contained in~$L'$. Properties of involutions in theses classes are as follows:
\begin{itemize}
 \item[$(2A)$] Let $u\in 2A$. Then 
 \begin{itemize}
  \item[$(2A1)$] $\operatorname{C}_L(u)\leqslant P_b$, $\operatorname{C}_L(u)/\operatorname{O}_2(\operatorname{C}_L(u))\cong\mbox{}^2B_2(q^2)$, and $|\!\operatorname{C}_L(u)|=q^{24}(q^2-1)(q^4+1)$; in particular, $u$ centralizes an element of order $5$ and does not centralize elements of order~$3$;
  \item[$(2A2)$] $u$ inverts an element of order $5$ and does not invert elements of order~$3$;
  \item[$(2A3)$] $\beta_3(u,L)=3$, $\beta_5(u,L)=2$.
 \end{itemize}

 \item[$(2B)$] Let $u\in 2B$. Then 
 \begin{itemize}
  \item[$(2B1)$] $\operatorname{C}_L(u)\leqslant P_a$, $\operatorname{C}_L(u)/\operatorname{O}_2(\operatorname{C}_L(u))\cong PSL_2(q^2)$, and $|\!\operatorname{C}_L(u)|=q^{20}(q^4-1)$; in particular, $u$ centralizes an element of order $3$ and does not centralize elements of order~$5$;
  \item[$(2B2)$] $u$ inverts elements of orders $3$ and~$5$.
  \item[$(2B3)$] $\beta_3(u,L)=\beta_5(u,L)=2$.
 \end{itemize}
 
\end{itemize}

\end{lem}
\begin{proof} The assertion about the number of classes of involutions in $L$ and assertions $(2A1)$ and $(2B1)$ follow from~\cite[(18.6)]{AschSei}. We consider a subgroup $L_0$ of $L$ isomorphic to~$\mbox{}^2F_4(2)'$. According to \cite{85Atlas}, $L_0$ also has two conjugacy classes $2A_0$ and $2B_0$ of involutions with centralizers of orders $10240=2^{11}\cdot 5$ and $1536=2^{9}\cdot 3$. By $(2A1)$ and $(2B1)$, we conclude that these involutions are representatives of classes $2A$ and $2B$ of $L$, respectively. Since $L=L'$ for $q^2>2$, we have $2A\cup 2B\subseteq L'$. Note that $L_0$~also has unique conjugacy classes of elements of orders $3$ and $5$ which we will denote by $3A_0$ and $5A_0$. Using~(\ref{StructConstFormula}) and the character table of $L_0$ in \cite{85Atlas}, we deduce that there are $5$ pairs of involutions $(u_1,u_2)\in 2A_0\times 2A_0$ whose product $u_1u_2$ equals a fixed element of $5A_0$. Every such pair generates a dihedral group of order~$10$, and every involution in this group inverts an element of order~$5$. We similarly deduce that there are $36$ pairs $(u_1,u_2)\in2B_0\times 2B_0$ whose product $u_1u_2$ is a fixed element of order $3$ in~$L_0$ and $25$ pairs whose product is a fixed element of order~$5$. Assertion $(2B2)$ and a part of assertion $(2A2)$ are proven. To finish the proof of $(2A2)$, it remains to see that the involutions in $2A$ do not invert an element of order~$3$. According to ~\cite[Theorem~4.7.3 and Table~4.7.3A]{98GorLySol} (see also \cite{LbkHome}), all elements of $L$ of order~$3$ are conjugate, and if $y\in L$ has order~$3$ then 
$$
C:=\operatorname{C}_L(y)\cong SU_3(q^2),\quad  N:=\operatorname{N}_L(\langle y\rangle)\cong SU_3(q^2):2,
$$
where $N$~is the extension of $SU_3(q^2)$ by a graph automorphism. Therefore, if $u\in 2A$ inverts $y$  then $u\in N\setminus C$. From~\cite[Theorem~4.9.2(f)]{98GorLySol}, it follows that all involutions in $ N\setminus C$ are conjugate by elements of~$C$ to the canonical graph involution and $\operatorname{C}_C(u)\cong Sp_2(q^2)$ by~\cite[Theorem~4.9.2(b2)]{98GorLySol}. But then $u$ centralizes an element of order~$3$, contrary to ~$(2A1)$. Assertion~$(2A2)$ is now fully proven. Due to the known properties of dihedral groups, $(2B2)$ implies~$(2B3)$ and $(2A2)$~implies both $\beta_5(u,L)=2$ and $\beta_3(u,L)>2$ for $u\in 2A$. To establish $(2A3)$ and complete the proof, it remains to show that $\beta_3(u,L)\leqslant 3$ for $u\in 2A$. To this end, we choose in a maximal subgroup $B_2(q^2):2$ a subgroup $H=B_2(2):2\cong \operatorname{Aut}(A_6)\cong S_6:2$ and an involution $u\in H\setminus S_6$. Such an involution is determined in $H$ up to conjugacy, and~$\operatorname{C}_H(u)$ contains an element of order~$5$, see~\cite{85Atlas}. Consequently, $u\in 2A$ by $(2A1)$ and $(2B1)$. In~\cite[Proposition~1.5]{21YangReVd}, it was proven that   $\beta_3(u,A_6)=3$, whence $\beta_3(u,L)\leqslant 3$. The proof is complete. 
 \end{proof}

\section{Properties of semisimple elements in $\mbox{}^2F_4(q^2)$} \label{SemisimpleProperties}

\begin{lem} \label{torilem} {\rm \cite[Theorem~3.1 and Table~III]{75Shin}} The maximal tori of 
$\mbox{}^2F_4(q^2)$ and their structure are as given Table~{\rm \ref{tori}}. 
\end{lem}

{
\begin{center}
\begin{table}
\caption
{The maximal tori of $\mbox{}^2F_4(q^2)$}\label{tori}
\begin{center}\begin{tabular}
{|l|c|r|}
\hline
\multirow{2}{*}{Torus} 
& \multirow{2}{*}{Structure}& Some\\
&&overgroups  \\
\hline\hline
\multirow{2}{*}{$T_1$} 
& \multirow{2}{*}{$\left(\mathbb{Z}_{q^2-1}\right)^2$}& \multirow{2}{*}{$Sp_4(q^2)$}\\
&&\\
\hline
\multirow{2}{*}{$T_2$} 
& \multirow{2}{*}{$\mathbb{Z}_{q^4-1}$}& {$Sp_4(q^2)\vphantom{A^{A^A}}$,}\\
& &$PGU_3(q^2)$\\
\hline
\multirow{2}{*}{$T_3$} 
& \multirow{2}{*}{$\mathbb{Z}_{q^2-1}\times \mathbb{Z}_{q^2-\sqrt{2}q+1}$}& 
\multirow{2}{*}{$\mbox{}^2B_2(q^2)\wr 2$}\\
&&\\
\hline
\multirow{2}{*}{$T_4$} 
& \multirow{2}{*}{$\mathbb{Z}_{q^2-1}\times \mathbb{Z}_{q^2+\sqrt{2}q+1}$}& \multirow{2}{*}{$\mbox{}^2B_2(q^2)\wr 2$}\\
&&\\
\hline
\multirow{2}{*}{$T_5$} 
& \multirow{2}{*}{$\mathbb{Z}_{q^4+1}$}& \multirow{2}{*}{$Sp_4(q^2)$}\\
&&\\
\hline
\multirow{2}{*}{$T_6$} 
& \multirow{2}{*}{$\left(\mathbb{Z}_{q^2-\sqrt{2}q+1}\right)^2$}& \multirow{2}{*}{$\mbox{}^2B_2(q^2)\wr 2$}\\
&&\\
\hline
\multirow{2}{*}{$T_7$} 
& \multirow{2}{*}{$\left(\mathbb{Z}_{q^2+\sqrt{2}q+1}\right)^2$}& \multirow{2}{*}{$\mbox{}^2B_2(q^2)\wr 2$}\\
&&\\
\hline
\multirow{2}{*}{$T_{8}$} 
& \multirow{2}{*}{$\left(\mathbb{Z}_{q^2+1}\right)^2$}& \multirow{2}{*}{$Sp_4(q^2)$}\\
&&\\
\hline
\multirow{2}{*}{$T_{9}$} 
& \multirow{2}{*}{$\mathbb{Z}_{q^4-q^2+1}$}& \multirow{2}{*}{$PGU_3(q^2)$}\\
&&\\
\hline
\multirow{2}{*}{$T_{10}$} 
& \multirow{2}{*}{$\mathbb{Z}_{q^4-\sqrt{2}q^3+q^2-\sqrt{2}q+1}$}& \\
&&\\
\hline
\multirow{2}{*}{$T_{11}$} 
& \multirow{2}{*}{$\mathbb{Z}_{q^4+\sqrt{2}q^3+q^2+\sqrt{2}q+1}$}& \\
&&\\
\hline
\end{tabular}\end{center}
\end{table}
\end{center}
}


{
\scriptsize
\begin{landscape} 
\begin{center}
\begin{table}
\caption
{Conjugacy classes of semisimple elements of $L=\mbox{}^2F_4(q^2)$}\label{Elements}
\begin{center}\begin{tabular}{|c|c|c|c|r|c|}\hline
\multirow{2}{*}{Notation} &Number& \multirow{2}{*}{$|\operatorname{C}_L(t_j)|$}& $T_i$, & {$\beta_3(t_j,L')$} or&\multirow{2}{*}{Justification} \\
 &of classes&&  $t_j\in T_i$ &$|t_j|$ nonprime& \\
\hline\hline

\multirow{2}{*}{$t_1$}&\multirow{2}{*}{$\displaystyle\frac{1}{2}(q^2-2)$}&\multirow{2}{*}{$q^4(q^2-1)^2(q^4+1)$} &\multirow{2}{*}{$T_1$}&\multirow{2}{*}{$2$}&$T_1\leqslant Sp_4(q^2)$, \\
      &&&&& Lemma~\ref{Sp4}\\
\hline

\multirow{2}{*}{$t_2$}&\multirow{2}{*}{$\displaystyle\frac{1}{2}(q^2-2)$}&\multirow{2}{*}{$q^2(q^2-1)(q^4-1)$}&\multirow{2}{*}{$T_1$}&\multirow{2}{*}{$2$}&$T_1\leqslant Sp_4(q^2)$, \\
 &&&&& Lemma~\ref{Sp4}\\
\hline

\multirow{2}{*}{$t_3$}  &\multirow{2}{*}{$\displaystyle\frac{1}{16}(q^2-2)(q^2-8)$}&\multirow{2}{*}{$(q^2-1)^2$}  &\multirow{2}{*}{$T_1$}&\multirow{2}{*}{$2$}&$T_1\leqslant Sp_4(q^2)$, \\
    &&&&& Lemma~\ref{Sp4}\\
\hline

\multirow{2}{*}{$t_4$}&\multirow{2}{*}{$1$}&\multirow{2}{*}{$q^6(q^4-1)(q^6+1)$}&\multirow{2}{*}{$T_2$}&\multirow{2}{*}{$1$}& {$|t_4|=3$,} \\
    &&&&& Lemma~\ref{u3}\\
\hline

\multirow{2}{*}{$t_5$}&\multirow{2}{*}{$\displaystyle\frac{1}{2}(q^2-2)$}&\multirow{2}{*}{$q^2(q^2+1)(q^4-1)$}&\multirow{2}{*}{$T_2$}&\multirow{2}{*}{$2$}& $T_2\leqslant Sp_4(q^2)$, \\
  &&&&& Lemma~\ref{Sp4} \\
\hline

\multirow{2}{*}{$t_6$}&\multirow{2}{*}{$\displaystyle\frac{1}{2}(q^2-2)$}&\multirow{2}{*}{$q^2(q^2+1)(q^4-1)$}&\multirow{2}{*}{$T_2$}&\multirow{2}{*}{$2$}& $T_2\leqslant Sp_4(q^2)$, \\
   &&&&& Lemma~\ref{Sp4} \\
\hline

\multirow{2}{*}{$t_7$}&\multirow{2}{*}{$\displaystyle\frac{1}{4}(q^2-\sqrt{2}q)$}&\multirow{2}{*}{$q^4(q^2-\sqrt{2}q+1)(q^2-1)(q^4+1)$}&\multirow{2}{*}{$T_{3}$}&\multirow{2}{*}{$2$}&$\alpha(t_7,L')=2$, \\
    &&&&&Lemma~\ref{alpphat7t9}  \\
\hline

\multirow{2}{*}{$t_8$}&\multirow{2}{*}{$\displaystyle\frac{1}{8}(q^2-2)(q^2-\sqrt{2}q)$}&\multirow{2}{*}{$(q^2-1)(q^2-\sqrt{2}q+1)$}&\multirow{2}{*}{$T_{3}$}&$|t_8|$ is& \multirow{2}{*}{Lemma~\ref{regularT3T4}} \\
   &&&&nonprime& \\
    \hline
    
\multirow{2}{*}{$t_9$}&\multirow{2}{*}{$\displaystyle\frac{1}{4}(q^2+\sqrt{2}q)$}&\multirow{2}{*}{$q^4(q^2+\sqrt{2}q+1)(q^2-1)(q^4+1)$}&\multirow{2}{*}{$T_{4}$}&\multirow{2}{*}{$2$}&$\alpha(t_9,L')=2$, \\
    &&&&&Lemma~\ref{alpphat7t9} \\
    \hline
    
\multirow{2}{*}{$t_{10}$}&\multirow{2}{*}{$\displaystyle\frac{1}{8}(q^2-2)(q^2+\sqrt{2}q)$}&\multirow{2}{*}{$(q^2-1)(q^2+\sqrt{2}q+1)$}&\multirow{2}{*}{$T_{4}$}& $|t_{10}|$ is&\multirow{2}{*}{Lemma~\ref{regularT3T4}}\\
    &&&&nonprime&  \\
    \hline
       
\multirow{2}{*}{$t_{11}$}&\multirow{2}{*}{$\displaystyle\frac{1}{16}q^2(q^2-\sqrt{2}q)$}&\multirow{2}{*}{$q^4+1$}&\multirow{2}{*}{$T_{5}$}&\multirow{2}{*}{$2$}&$T_{5}\leqslant Sp_4(q^2)$, \\
    &&&&&Lemma~\ref{Sp4} \\ 
    \hline
    
 \multirow{2}{*}{$t_{12}$}&\multirow{2}{*}{$\displaystyle\frac{1}{96}(q^2-\sqrt{2}q-4)(q^2-\sqrt{2}q)$}&\multirow{2}{*}{$(q^2-\sqrt{2}q+1)^2$}&\multirow{2}{*}{$T_{6}$}&\multirow{2}{*}{$2$}&$\alpha(t_{12},L')=2$, \\
    &&&&& Lemma~\ref{alpphat12t13}\\   \hline
    
\multirow{2}{*}{$t_{13}$}&\multirow{2}{*}{$\displaystyle\frac{1}{96}(q^2+\sqrt{2}q-4)(q^2+\sqrt{2}q)$}&\multirow{2}{*}{$(q^2+\sqrt{2}q+1)^2$}&\multirow{2}{*}{$T_{7}$}&\multirow{2}{*}{$2$}&$\alpha(t_{13},L')=2$, \\
    &&&&& Lemma~\ref{alpphat12t13} \\  \hline
    
 \multirow{2}{*}{$t_{14}$}&\multirow{2}{*}{$\displaystyle\frac{1}{48}(q^2-2)(q^2-8)$}&\multirow{2}{*}{$(q^2+1)^2$}&\multirow{2}{*}{$T_{8}$}&\multirow{2}{*}{$2$}&$T_{8}\leqslant Sp_4(q^2)$,  \\
    &&&&&Lemma~\ref{Sp4} \\  \hline
    
 \multirow{2}{*}{$t_{15}$}&\multirow{2}{*}{$\displaystyle\frac{1}{6}(q^2-2)(q^2+1)$}&\multirow{2}{*}{$q^4-q^2+1$}&\multirow{2}{*}{$T_{9}$}&\multirow{2}{*}{$\leqslant 3$}&$T_{9}\leqslant PGU_3(q^2)$, \\
    &&&&& Lemma~\ref{U3} \\   \hline
    
 \multirow{2}{*}{$t_{16}$}&\multirow{2}{*}{$\displaystyle\frac{1}{12}(q^4-\sqrt{2}q^3+q^2-\sqrt{2}q)$}&\multirow{2}{*}{$q^4-\sqrt{2}q^3+q^2-\sqrt{2}q+1$}&\multirow{2}{*}{$T_{10}$}&\multirow{2}{*}{$2$}& 
 \multirow{2}{*}{Lemma~\ref{toriT10T11}}\\
    &&&&&  \\    \hline
 
 \multirow{2}{*}{$t_{17}$}&\multirow{2}{*}{$\displaystyle\frac{1}{12}(q^4+\sqrt{2}q^3+q^2+\sqrt{2}q)$}&\multirow{2}{*}{$q^4+\sqrt{2}q^3+q^2+\sqrt{2}q+1$}&\multirow{2}{*}{$T_{11}$}&\multirow{2}{*}{$2$}&
 \multirow{2}{*}{Lemma~\ref{toriT10T11}}\\
    &&&&& \\
\hline

\end{tabular}\end{center}
\end{table}
\end{center}

\end{landscape}
}

We note that the maximal tori contained in the subsystem subgroups of maximal rank $B_2(q^2)\cong Sp_{2}(q^2)$, $PGU_3(q^2)$, and ${}^2B_2(q^2)\times {}^2B_2(q^2)$ are also maximal tori of  $\mbox{}^2F_4(q^2)$. From the classification of such maximal tori \cite[Theorem~3.1 and Table~III]{75Shin}, the information in the last column of Table~{\rm \ref{tori}} holds.

Table~\ref{Elements} collects all the necessary information about the conjugacy classes of nonidentity semisimple elements of $L=\mbox{}^2F_4(q^2)$. 
According to ~\cite[Theorem~3.2 and Table~IV]{75Shin}, depending on the structure and order of the centralizer, the classes are split into $17$ groups with a common notation $t_j$ for a representative of each group, $j=1,\dots,17$. 
The information in the first three columns of Table~\ref{Elements} is taken from~\cite[Columns 1---3 of Table~IV]{75Shin}. 
A torus $T_i$ containing $t_j$ indicated in the fourth column of Table~\ref{Elements} can be easily found from both the generic form of $t_j$ in~\cite[Column 1 of Table~IV]{75Shin} and the generic form of elements in each torus~$T_i$, see~\cite[The list in the end of Section (3.1)]{75Shin}. 
The key r\^ole in this paper is played by the information in the fifth and sixth columns of Table~\ref{Elements}. 
The fifth column gives either the precise value or an upper bound for $\beta_3(t_j,L')$, or else states that the order of $t_j$ is not prime which spares us from its calculation. 
The last column justifies the fifth one by giving a relevant reference.

\begin{lem} \label{regularT3T4} Let  $x$~be a regular element of one of the tori $T_{3}$ or $T_{4}$ in  $L=\mbox{}^2F_4(q^2)$,  see Table~{\rm \ref{tori}}. Then $|x|$ is not prime. 
\end{lem}
\begin{proof} The numbers 
$$
q^2-1,\quad q^2-\sqrt{2}q+1,\quad\text{and}\quad q^2+\sqrt{2}q+1
$$ 
are pairwise coprime. Hence, $T_3$ is a cyclic group of order $(q^2-1)(q^2-\sqrt{2}q+1)$. It is contained in a subsystem subgroup of maximal rank ${}^2B_2(q^2)\times {}^2B_2(q^2)$ whose normalizer is a maximal subgroup   ${}^2B_2(q^2)\wr 2$. Since the order of every $2'$-element of ${}^2B_2(q^2)$ divides one of the numbers  $q^2-1$, $q^2-\sqrt{2}q+1$, or $q^2+\sqrt{2}q+1$ \cite[Ch.~XI, Theorem~3.10]{HupBlackIII}, it is clear that the projection of $T_3$ to one of the direct factors of ${}^2B_2(q^2)\times {}^2B_2(q^2)$ has order $q^2-1$, while the projection to the other factor has order $(q^2-\sqrt{2}q+1)$. Consequently, the elements of $T_3$ whose order divides $q^2-1$ are cointained in one of the direct factors, and those of order dividing  $q^2-\sqrt{2}q+1$ are in the other one. Thus, if we assume that $x\in T_{3}$ has prime order then either $|x|$ divides $q^2-1$, and so $\operatorname{C}_L(x)$ contains a subgroup isomorphic to $\mathbb{Z}_{q^2-1}\times {}^2B_2(q^2)$, or $|x|$ divides $q^2-\sqrt{2}q+1$, and then $\operatorname{C}_L(x)$ contains a subgroup isomorphic to $\mathbb{Z}_{q^2-\sqrt{2}q+1}\times {}^2B_2(q^2)$. In either case, $\operatorname{C}_L(x)\ne T_3$ and $x$ is not a regular element of $T_3$. It can be similarly shown that the regular elements in $T_4$ have nonprime orders. 
\end{proof}

\begin{lem} \label{toriT10T11} Let  $x$~be an element of $L=\mbox{}^2F_4(q^2)$ of prime order contained in one of the tori $T_{10}$ or $T_{11}$ from Table~{\rm \ref{tori}}. Then 
$
\beta_3(x,L')= 2.$

\end{lem} 

\begin{proof}
The decompositions (\ref{phi4_12_dec}) with $l=q^2$ yield
\begin{align*}
 q^4+1 &= (q^2-\sqrt{2}q+1)(q^2+\sqrt{2}q+1),\\   
q^8-q^4+1&=(q^4-\sqrt{2}q^3+q^2-\sqrt{2}q+1)(q^4+\sqrt{2}q^3+q^2+\sqrt{2}q+1).
\end{align*} 
This, together with Lemma~\ref{maximal}, readily implies that the odd prime $|x|$ which divides 
$$
 |T_{10}||T_{11}|=(q^4-\sqrt{2}q^3+q^2-\sqrt{2}q+1)(q^4+\sqrt{2}q^3+q^2+\sqrt{2}q+1)
$$ 
is greater than $3$ and divides the order of no maximal subgroup of $\mbox{}^2F_4(q^2)$ except, possibly,   
$$
 \operatorname{N}_L(T_{10})\cong (\mathbb{Z}_{q^4-\sqrt{2}q^3+q^2-\sqrt{2}q+1}):12,\quad \operatorname{N}_L(T_{11})\cong (\mathbb{Z}_{q^4+\sqrt{2}q^3+q^2+\sqrt{2}q+1}):12
$$ 
$$
 \text{and}\quad\mbox{}^2F_4(q_0^2),\quad \text{where}\quad  q=q_0^r.
$$ 

Assume that $x$ is not conjugate to any element of any subgroup $\mbox{}^2F_4(q_0^2)$. The numbers   $$
 q^4-\sqrt{2}q^3+q^2-\sqrt{2}q+1\quad\text{and}\quad q^4+\sqrt{2}q^3+q^2+\sqrt{2}q+1
$$ 
are coprime. Hence, $x$ is contained in precisely one maximal subgroup which is either  $\operatorname{N}_L(T_{10})$ or  $\operatorname{N}_L(T_{11})$. There is a conjugate $x^g$ that does not commute with $x$ and so is not contained in the same maximal subgroup as~$x$. In this case, $\langle x,x^g\rangle=L$. Therefore,  $L=L'$,
$$
\beta_3(x,L')\leqslant\alpha(x,L')=2,
$$ 
and $\beta_3(x,L')=2$, since $|x|\ne3$.

If $x$ is contained in a subgroup  $L_0=\mbox{}^2F_4(q_0^2)$ then arguing by induction we obtain $\beta_3(x,L_0')= 2$, whence $\beta_3(x,L')= 2$. 
\end{proof}

\section{Generation of ${}^2F_4(q^2)'$ by conjugate elements \\ and the proof of Theorem~\ref{thm:2F4}}

This section deals with proving Theorem~\ref{thm:2F4}. 
 For elements of orders 
$3$ and $5$, Theorem~{\rm \ref{thm:2F4}} is refined by Propositions~\ref{alpha3} and~\ref{alpha5} below which will play a key r\^ole it its proof owing to the following assertion:

\begin{lem} \label{GuestTeorem} Let $L=\mbox{}^2F_4(q^2)'$ and let $x\in \operatorname{Aut}(L)$~have odd prime order. Then $x$ together with its conjugate in $\langle x,L\rangle$ generates a nonsolvable subgroup.
\end{lem}
\begin{proof} This is a consequence of \cite[Theorem~A*]{10Guest}.
\end{proof}

\begin{lem} \label{u3} {\rm \cite[Theorem~4.7.3 and Table~4.7.3A]{98GorLySol}}
The group $L=\mbox{}^2F_4(q^2)$ has a unique conjugacy class of elements of order $3$. The centralizer  $\operatorname{C}_L(x)$of every such $x$ is isomorphic to  $SU_3(q^2)$ and has order~$q^6(q^4-1)(q^6+1)$.
\end{lem}

Lemma~\ref{u3} and columns $2$ and $3$ of Table~\ref{Elements} imply that the elements of order~$3$ in $\mbox{}^2F_4(q^2)$~are precisely the conjugates of~$t_4$.

\begin{prop}
 \label{alpha3}
Let $L=\mbox{}^2F_4(q^2)'$ and let ${x\in L}$~be an element of order~$3$. Then ${\alpha(x,L)=2}$.
\end{prop}
\begin{proof} It is known \cite[Theorem~1]{95Mal} that  $L$ is $(2,3)$-generated, i.\,e. generated by an involution $u$ and an element $y$ of order $3$. Consequently \cite[Corollary~1.2]{99LubMal}, $L$~is generated by two elements $y$ and $y^u$ of order~$3$. Now, $y$ and $y^u$ are conjugate to $x$ by Lemma~\ref{u3} and the claim follows.
\end{proof}

One of the key facts is that an assertion similar to  Proposition~\ref{alpha3} holds for elements of order~ $5$. We first establish an analog of Lemma~\ref{u3} for such elements.

\begin{lem} \label{u5}
The group
$L=\mbox{}^2F_4(q^2)$ has a unique conjugacy class of elements of order $5$. The centralizer of such an element has order $q^4(q^2+\delta \sqrt{2}q+1)(q^2-1)(q^4+1)$, where $\delta=\pm1$ is chosen so that $q^2+\delta \sqrt{2}q+1$ is divisible by ~$5$. In the notation of Table~{\rm \ref{Elements}}, the elements of order $5$ are conjugate to an elements of type $t_7$ if $\delta=1$ and $t_9$ if $\delta=-1$.
\end{lem}
\begin{proof}
Since $q^2\equiv \pm 2 \pmod{5}$, we have
$$
q^2(q^2-1)(q^6+1)(q^8-q^4+1)\not\equiv 0\pmod{5}
$$
and $|L|_5=(q^4+1)^2_5$ by (\ref{og}). Now, $q^4+1=(q^2+\sqrt{2}q+1)(q^2-\sqrt{2}q+1)$ and
we define $\delta =\pm 1$ so that $q^2+\delta \sqrt{2}q+1 \equiv 0 \pmod{5}$. It can be seen that
\begin{align*}
  \delta = \phantom{-}1 \ \ \Leftrightarrow \ \ &  n \equiv 0,3 \pmod{4},  \\
  \delta = -1           \ \ \Leftrightarrow \ \ &  n \equiv 1,2 \pmod{4}.
\end{align*}
It follows from Lemma~\ref{torilem} that $L$ contains a maximal torus of shape 
$$
\mathbb{Z}_{q^2+\delta\sqrt{2}q+1}\times \mathbb{Z}_{q^2+\delta\sqrt{2}q+1}
$$ 
which implies that the $5$-Sylow subgroup of $G$ is isomorphic to $\mathbb{Z}_{5^k}\times \mathbb{Z}_{5^k}$, where $5^k=(q^4+1)_5$. Hence, all elements of order $5$ in a given $5$-Sylow subgroup $U$ of $L$ lie in the subgroup $\Omega_1(U)\cong \mathbb{Z}_{5}\times \mathbb{Z}_{5}$. Now the Sylow theorem implies that all subgroups isomorphic to $\mathbb{Z}_{5}\times \mathbb{Z}_{5}$ are conjugate. Since $G$ contains a subgroup $H\cong\mbox{}^2F_4(2)'$  whose $5$-Sylow subgroup is $\mathbb{Z}_{5}\times \mathbb{Z}_{5}$, it follows that every element of order $5$ of $L$ is conjugate to one of $H$. However, $H$ has a unique conjugacy class of elements of order $5$, see~\cite{85Atlas}, hence, so does $L$.

Let $t\in L$ be a fixed element of order $5$. It lies in a unique conjugacy class of such elements by Lemma~\ref{u5}. 
Since a $5$-Sylow subgroup of $G$ is abelian,
$|\!\operatorname{C}_L(t)|$ must be divisible by $|L|_5$. We see from Table~\ref{Elements}
that the only conjugacy class representatives with this property are
$t_7$, $t_9$ with centralizers of order $q^4(q^2\pm \sqrt{2}q+1)(q^2-1)(q^4+1)$ and $t_{12}$, $t_{13}$
with centralizers of order $(q^2\pm \sqrt{2}q+1)^2$. We rule out the last two possibilities. From the structure of the maximal subgroup $H=\mbox{}^2B_2(q^2)\wr 2$ of $G$, we see that if $t$ lies in
a maximal torus of $\mbox{}^2B_2(q^2)$ of order $q^2+\delta \sqrt{2}q+1$ then $\!\operatorname{C}_H(t)$ has order divisible by
$$
(q^2+\delta\sqrt{2}q+1)\cdot|\mbox{}^2B_2(q^2)|=q^4(q^2+\delta \sqrt{2}q+1)(q^2-1)(q^4+1).
$$
Since all elements of order $5$ are conjugate, $t$ cannot have
centralizers of order $(q^2\pm \sqrt{2}q+1)^2$. Thus, $t$ is conjugate to either $t_7$ or $t_9$ depending on whether
$\delta = -1$ or $\delta = 1$. 
\end{proof}

We will require an analog of Proposition~\ref{alpha3} for elements of order~$5$. First, we show that $L$ is $(5,5)$-generated, that is generated by a pair of elements of order~$5$. Note that these elements will necessarily be conjugate by Lemma~\ref{u5}. Similarly to the results in 
\cite{95Mal} (which imply Proposition~\ref{alpha3}), our proof relies on calculations in  
the \textsf{CHEVIE} package for \textsf{Maple} \cite{Chevie} which contains a generic character table of~$L$. In this section, we just formulate the necessary assertion and cite the lemma that implies it. The lemma itself is proved in Section \ref{SecComp}.

\begin{prop}
\label{alpha5}
Let $L=\mbox{}^2F_4(q^2)'$ and let ${x\in L}$~be an element of order~$5$. Then ${\alpha(x,L)=2}$.
\end{prop}
\begin{proof}
In view of Lemma~\ref{u5}, this is a particular case of Lemma~\ref{alpphat7t9}.    
\end{proof}

\begin{proof}[Proof of Theorem~{\rm \ref{thm:2F4}}]
We fix an automorphism $x$ of $L$ of prime order. First, assume that $|x|$ is odd. By Lemma~\ref{GuestTeorem}, some subgroup of the form $\langle x,x^g\rangle$, where $g\in L$, is nonsolvable. Since every nonabelian simple groups of order not divisible by $3$ is isomorphic to a Suzuki group \cite[Ch. II, Corollary 7.3]{76Gllaub}, the subgroup $\langle x,x^g\rangle$ contains an element   $y\in L$ of order~$3$ or~$5$. If $x$ is an involution then, by Lemma~\ref{involutions}, $\langle x,x^g\rangle$ also 
 contains an element $y$ of order~$3$ or~$5$ for some $g\in L$. Propositions~\ref{alpha3} and~\ref{alpha5} give  $\alpha(y,L)=2$ which yields $\alpha(x,L)\leqslant 4$. \end{proof}

\section{An application to studying the $\pi$-radical}

In this section, we give proofs of Theorem~{\rm \ref{thm:2F4beta}} and Corollary~{\rm \ref{BaerSuzuki}}. As in the proof of Theorem \ref{thm:2F4}, we will require some results that are obtained using computer calculations.

\begin{proof}[Proof of Theorem~{\rm \ref{thm:2F4beta}}]
By Theorem~\ref{thm:2F4}, we have 
$$
\beta_s(x,L)\leqslant\alpha(x,L)\leqslant 4
$$
for every divisor $s$ of $|L|$ and every automorphism $x$ of $L$ of prime order. It remains to prove that  $\beta_3(x,L)\leqslant 3$.  Three cases are possible.

\emph{Case}~$(i)$: $x$~is an involution. Then the claim follows from Lemma~\ref{involutions}.

\emph{Case}~$(ii)$: $x$~is an inner automorphism of odd order. Then $x$ is induced by conjugation with a semisimple element that is conjugate to one of $t_1$--$t_{17}$ in   Table~\ref{Elements} as follows from \cite{75Shin}. For every $t_j$ of prime order, we have $\beta_3(t_j,L)\leqslant 3$ by Table~\ref{Elements}. A justification of this fact, together with a relevant reference, is given in the last column of Table~\ref{Elements}.

\emph{Case}~$(iii)$: $x$ is not an inner automorphism. Then  $x$ is conjugate in $\langle x,L\rangle$ to a canonical field automorphism. Hence, we may assume that $x$ itself is 
a field automorphism~\cite[Proposition~4.9.1]{98GorLySol}. By Lemma~\ref{maximal}, $x$ acts in an obvious way (i.\,e. induces a field automorphism) on a maximal parabolic subgroup~$P_a$, see Lemma~\ref{maximal}, and also induces a nontrivial automorphism $\overline{x}$ on the quotient $\overline{P_a}\cong PSL_2(q^2)$ of~$P_a$ by its solvable radical. Now, Lemma~\ref{L_2} implies 
$$
\beta_3(x,L)\leqslant\beta_3(\overline{x\vphantom{P_a}},\overline{P_a})\leqslant 3.
$$
The proof is complete.
\end{proof}

\begin{proof}[Proof of Corollary~{\rm \ref{BaerSuzuki}}]
Given a group $G$, we will write $G\in{{\mathscr B}{\mathscr S}}_{\pi}^{m}$ if $G$ has the following property: a conjugacy class $D$ of $G$ is contained in~$\mathrm{O}_\pi(G)$ (equivalently, generates a $\pi$-subgroup) if and only if every $m$ elements of $D$ generate a $\pi$-subgroup. 

Suppose that the claim is false. By ~\cite[Theorem~1]{11Re}, we have ${r\geqslant 3}$. Set
\begin{multline*}
\mathscr{F}=\mathscr{S}\cup\{A_n\mid n\geqslant 5\}\cup
\{L^\varepsilon_n(q)\mid n\geqslant 2,\ \varepsilon=\pm,\  q\text{~is a prime power}
\}\cup\\
\{{}^2B_2(2^{2n+1})\mid n\geqslant 1\}\cup\{{}^2G_2(2^{3n+1})\mid n\geqslant 1\}\cup
\{{}^2F_4(2^{2n+1})'\mid n\geqslant 0\}\cup\\
\{G_2(q),{}^3D_4(q))\mid  q\text{~is a prime power}\},
\end{multline*}
where $\mathscr{S}$~is the set of $26$ sporadic groups. We consider the class $\mathscr{X}$ of all finite groups whose every nonabelian composition factor is either a group in  $\mathscr{F}$ or a $\pi$-group. Then 
${\mathscr{X}\setminus \mathscr{BS}_\pi^m\ne\varnothing}$, 
since the inclusion 
${\mathscr{X}\subseteq \mathscr{BS}_\pi^m}$ would imply the validity of the corollary we are proving.

According to \cite[Lemma~7]{11Re}, a group $G$ in $\mathscr{X}\setminus \mathscr{BS}_\pi^m$ of the least order contains a normal nonabelian simple subgroup $L$ and an element $x$ of prime order with the following properties:
\begin{itemize}
\item $\mathrm{C}_G(L)=1$, and so we may assume that $L
\leqslant G\leqslant\mathrm{Aut}(L)$;
\item $G=\langle L,x\rangle$;
\item every $m$ conjugates of $x$ generate a $\pi$-subgroup of $G$;
\item $L$ is neither a $\pi$- nor a $\pi'$-group, and so we may assume that $L\in\mathscr{F}$.
\end{itemize}
Let $s$~be the least prime divisor of $|L|$ not contained in $\pi$. Then either $r$ divides $|L|$ and $s=r$ or $r$ does not divide $|L|$ and $s>r$. According to Theorem~\ref{thm:2F4beta} and the assertions  \cite[Prorosition~1.5]{21YangReVd}, \cite[Theorem~1]{22YangWuRe}, \cite[Theorem~1]{23YangWuReVd}, and \cite[Theorem~2]{Re}, we have 
\begin{equation*}
        \beta_s(x,L)\leqslant m,
\end{equation*}
i.\,e., there are elements ${g_1,\dots,g_m\in L}$ such that $|\langle{ x^{g_1},\dots,x^{g_m}}\rangle|$ is divisible by $s$. 
But then $\langle{ x^{g_1},\dots,x^{g_m}}\rangle$~is not a $\pi$-subgroup contrary to the fact that every $m$ conjugates of $x$ generate a $\pi$-subgroup.
\end{proof}

\section{Computer calculations}\label{SecComp}

In this section, we prove the following assertion:

\begin{lem} \label{alpphat7t9}\label{alpphat12t13} Let  $x$~be one of the semisimple elements $t_7$, $t_9$, $t_{12}$, or $t_{13}$ of~$L={}^2F_4(l)'$, see Table~{\rm \ref{Elements}}. Then $\alpha(x,L)=2$. 
\end{lem}

We sketch our strategy of proof. First of all, the Tits group $\mbox{}^2F_4(2)'$ can be checked individually using \textsf{GAP}. The following stronger fact holds:
\begin{prop}
    \label{TitsGen}
Let $L=\mbox{}^2F_4(2)'$ and let $x\in L$. Then $\alpha(x,L)=2$ if $x$ has order greater than $2$  and  $\alpha(x,L)=3$ if $x$ has order $2$.  
\end{prop}
\begin{proof}
    The Tits group 
$\mbox{}^2F_4(2)'$ of order 
$17971200$ can be constructed in \textsf{GAP}, for example, using the matrix generators given in 
\cite{AtlRep}. It has $22$ conjugacy classes. For a  representative $x$ of each nonidentity class, we can find in \textsf{GAP} explicitly one or two conjugates that generate the whole group together with $x$.
\end{proof}

Consequently, we may assume that $L=\mbox{}^2F_4(l)$, where $l=2^{2n+1}>2$. We have
\begin{equation}\label{og}
|L|=l^{12}(l^6+1)(l^4-1)(l^3+1)(l-1).
\end{equation}
Under our current assumptions, the set $\Pi_{12}(l)$ of primitive prime divisors of $l^{12}-1$ is nonempty. It is contained in the set of all divisors of 
$$
\Phi_{12}(l)=l^4-l^2+1=(l^2+\sqrt{2l^3}+l+\sqrt{2l}+1)(l^2-\sqrt{2l^3}+l-\sqrt{2l}+1)
$$  
and splits into the disjoint nonempty subsets $\Pi_{12}^+(l)$ and  $\Pi_{12}^-(l)$, where the elements of $\Pi_{12}^\varepsilon(l)$ divide  $l^2+\varepsilon\sqrt{2l^3}+l+\varepsilon\sqrt{2l}+1$.
\big(The fact that $\Pi_{12}^\varepsilon(l)$ is nonempty unless $l=2$ and $\varepsilon=-1$ follows from~\cite[Theorem~2]{Grechkoseeva}, and we are assuming that $l>2$.\big) For every    $\varepsilon=\pm$, we fix a prime $r_\varepsilon\in \Pi_{12}^\varepsilon(l)$. Also, let $T_\varepsilon$ denote a torus of order $l^2+\varepsilon\sqrt{2l^3}+l+\varepsilon\sqrt{2l}+1$ in~$L$, i.\,e. $T_+=T_{11}$ and $T_-=T_{10}$ in the notation of Table~\ref{tori}.

It follows from Lemma~\ref{maximal} that the only maximal subgroup of $L$ of order divisible by $r_\varepsilon$ can be the normalizer $\operatorname{N}_L(T_\varepsilon)$. Note that the numbers 
$$
l-1,\quad  l^2+1 =(l-\varepsilon \sqrt{2l}+1)(l+\varepsilon \sqrt{2l}+1), \text{~and} 
$$
$$
l^4-l^2+1=(l^2+\varepsilon\sqrt{2l^3}+l+\varepsilon\sqrt{2l}+1)(l^2-\varepsilon\sqrt{2l^3}+l-\varepsilon\sqrt{2l}+1)
$$
are pairwise coprime. Therefore, $|\operatorname{N}_L(T_\varepsilon)|=12|T_\varepsilon|$ is not not divisible by the prime divisors of $|T_i|$ greater than $3$, where $i=3,4,6,7$. The elements $t_7$, $t_9$, $t_{12}$ and $t_{13}$ are contained in these tori and have order greater than~$3$ (the elements of order $3$ are conjugate to~$t_4$).
Therefore, if we show that there is a pair $(x,y)$ of elements conjugate to $t_j$, $j=7,9,12,13$, whose product $z=xy$ has order divisible by $r_\varepsilon$, this pair will generate the whole of $L$ and we will have $\alpha(x,L)=2$.
We show this by calculating the suitable class multiplication coefficients $(x^L,x^L,z^L)$, where $z$ is chosen from  $T_\varepsilon$ so that $|z|$ is divisible by $r_\varepsilon$. 
It can be seen from Table~\ref{Elements} that the only semisimple elements with centralizers of order divisible by $r_\varepsilon$ (in particular, the proper elements of order~$r_\varepsilon$) are those conjugate to $t_{16}$ or $t_{17}$ according as $\varepsilon=-1$ or $\varepsilon=1$. 
In Lemma~\ref{StructConst} below, we show that  $(x^L,x^L,z^L)>0$. This will prove Lemma~\ref{alpphat12t13}.   

\begin{lem} \label{StructConst} Every multiplication coefficient $(t_j^L,t_j^L,t_k^L)$, where $j=7,9,12,13$ and $k=16,17$, is positive.
\end{lem}

\begin{proof}
To calculate the required class multiplication coefficient
\begin{equation}\label{cmcs}
f_{j,k}=(t_{j}^L,t_{j}^L,t_{k}^L),    
\end{equation}
\noindent
we use the \textsf{CHEVIE} package for \textsf{Maple} \cite{Chevie} which contains a generic character table of $L$, although not all irreducible characters are included in it.
The fact that some generic irreducible characters of $\mbox{}^2F_4(l)$ are missing from
\textsf{CHEVIE} does not affect the validity of calculation.
This is because comparing the degrees of irreducible characters given in  \cite{LbkHome} we see that
all the missing characters have degrees divisible by $l^4-l^2+1$ and thus vanish on all elements of order divisible by $r_\varepsilon$ \cite[Theorem~(8.17)]{06Isa}.

The element $t_i$ for $i=7$,\,$9$,\,$12$,\,$13$,\,$16$,\,$17$ has \textsf{CHEVIE} class type number $35$,\allowbreak\,$40$,\allowbreak\,$46$,\allowbreak\,$47$,\allowbreak\,$50$,\allowbreak\,$51$, respectively. So we can calculate the class multiplication coefficient~(\ref{cmcs}), say, for $(j,k)=(7,16)$ by running the \textsf{CHEVIE} command

\medskip
\noindent\verb|> ClassMult(T,35,35,50);|
\medskip

\noindent
where \verb|T| is the generic character table of $\mbox{}^2F_4(q^2)$, and similarly for other values of $j$ and $k$. This results in the following eight polynomials in $q$:

\begin{equation}
\begin{array}{r@{\,}l@{\qquad}r@{\,}l}

  f_{7,16}(q)&=f_1(q) + \sqrt{2}f_2(q),  &  f_{7,17}(q)&=f_3(q) + \sqrt{2}f_4(q), \\[5pt]
  f_{9,16}(q)&=f_3(q) - \sqrt{2}f_4(q),  &  f_{9,17}(q)&=f_1(q) - \sqrt{2}f_2(q), \\[5pt]
  f_{12,16}(q)&=f_5(q) + \sqrt{2}f_6(q),  &  f_{12,17}(q)&=f_7(q) + \sqrt{2}f_8(q), \\[5pt]
  f_{13,16}(q)&=f_7(q) - \sqrt{2}f_8(q),  &  f_{13,17}(q)&=f_5(q) - \sqrt{2}f_6(q),
\end{array}
\end{equation}
where

\begin{multline*}
f_1 = q^{28}+5q^{26}+2q^{24}-17q^{22}+61q^{18}+11q^{16}-68q^{14}+11q^{12}+61q^{10}\\
\shoveright{-17q^{6}+2q^{4}+5q^{2}+1,}\\
\shoveleft{f_2 = 2q^{27}+4q^{25}-5q^{23}-13q^{21}+24q^{19}+36q^{17}-39q^{15}-39q^{13}+36q^{11}}\\
\shoveright{+24q^{9}-13q^{7}-5q^{5}+4q^{3}+2q,}\\
\shoveleft{f_3=q^{28}+5q^{26}+2q^{24}-9q^{22}-8q^{20}-3q^{18}+19q^{16}+28q^{14}+19q^{12}-3q^{10}}\\
\shoveright{-8q^{8}-9q^{6}+2q^{4}+5q^{2}+1,}\\
\shoveleft{f_4=2q^{27}+4q^{25}-3q^{23}-7q^{21}-4q^{19}+4q^{17}+19q^{15}+19q^{13}+4q^{11}-4q^{9}}\\
\shoveright{-7q^{7}-3q^{5}+4q^{3}+2q,}\\
\shoveleft{f_5 =  q^{44}+15q^{42}+14q^{40}-77q^{38}-63q^{36}+220q^{34} +65q^{32}-717q^{30}+ 14q^{28}}\\
\shoveright{+1807q^{26}+137q^{24}-2232q^{22} +137q^{20}+1807q^{18}+14q^{16}}\\
\shoveright{-717q^{14}+65q^{12}+220q^{10}-63q^{8}-77q^{6}+14q^{4}+15q^{2}+1,}\\
\shoveleft{ f_6 = 4q^{43}+16q^{41}-16q^{39}-76q^{37}+48q^{35}+184q^{33}-244q^{31}-472q^{29}+760q^{27}}\\
\shoveright{+1036q^{25}-1168q^{23}-1168q^{21}+1036q^{19}+760q^{17}-472q^{15}}\\
\shoveright{-244q^{13}+184q^{11}+48q^{9}-76q^{7}-16q^{5}+16q^{3}+4q,}\\
\shoveleft{ f_7 =  q^{44}+15q^{42}+14q^{40}-77q^{38}-63q^{36}+220q^{34}+193q^{32}-205q^{30}-370q^{28}}\\
\shoveright{ -241q^{26}+457q^{24} +904q^{22} +457q^{20}-241q^{18}-370q^{16}-205q^{14}}\\
\shoveright{+193q^{12}+220q^{10}-63q^{8}-77q^{6}+14q^{4}+15q^{2}+1,}\\
\shoveleft{ f_8 = 4q^{43}+16q^{41}-16q^{39}-76q^{37}+48q^{35}+200q^{33}  -4q^{31}-232q^{29}-248q^{27}}\\
\shoveright{+28q^{25} +560q^{23} +560q^{21}  +28q^{19}-248q^{17}-232q^{15}-4q^{13}}\\
+200q^{11}+48q^{9}-76q^{7}-16q^{5}+16q^{3}+4q.
\end{multline*}
We only need to show that $f_{j,k}(q)>0$ for $q\geqslant \sqrt{2}$; $j=7,9,12,13$, $k=16,17$.
It can be checked that in fact
all derivatives satisfy $f^{(i)}_{j,k}(\sqrt{2})>0$, where $i=0,\ldots,28$ for $j=7,9$, $k=16,17$
and $i=0,\ldots,44$ for $j=12,13$, $k=16,17$.
Hence, all coefficients in the Taylor expansion of $f_{j,k}(q)$ into the powers of  $q-\sqrt{2}$ are strictly positive, and so  $f_{j,k}(q)$ is strictly increasing as a function of $q$ on the interval $[\sqrt{2},+\infty)$. The claim follows from these remarks.
\end{proof}

\section{Final remarks and open problems}

The original aim of this research was to estimate the value of $\beta_s(x,L)$ in the case $L={}^2F_4(q^2)'$ and then apply it to the study of Conjecture~\ref{ConjBS}. Theorem~\ref{thm:2F4} can be viewed as a byproduct of our results.  

Comparing Theorem~\ref{thm:2F4} and its refinements for particular conjugacy classes given in Propositions ~\ref{alpha3}, ~\ref{alpha5}, and Lemma~
 \ref{alpphat7t9}, as well as the precise result for the Tits group in Proposition~\ref{TitsGen}, it is natural to ask what the exact value of $\alpha(x,L)$ is for every $x$, or at least for every $x$ of prime order. In particular, can we improve  in Theorem~\ref{thm:2F4} the estimate for $\alpha(x,L)$ from $4$ to $3$? 

The estimate $\alpha(x,L)\leqslant3$ would certainly be best possible, because a pair of involutions always generates a solvable group. Thus, $\alpha(x,L)\geqslant 3$ for an involution $x\in L$. It would be natural to look for a counterexample to the conjecture
$\alpha(x,L)\leqslant 3$ among the involutions. 

Most likely, such a counterexample does not exist, including among the involutions, as some known results speak in favour of it. Thus, T.\,Weigel   \cite[Theorem~B]{92Weig} proved that every exceptional group $L$ of Lie type can be generated by three involutions (just as every finite simple group other than $PSU_3(3)$, \cite[Theorem~A]{94MalSaxWeig}). Earlier, Ya.\,N.\,Nuzhin~\cite{90Nuzh} proved that a group of Lie type over a finite field of characteristic~$2$ can be generated by three involutions two of which commute if and only if it has type other than $A_2$, ${}^2A_2$, $A_3$, or ${}^2A_3$. For many groups,  ${}^2F_4(q^2)'$ including, an explicit triple of such involutions is given in~\cite{90Nuzh}. G.~Malle \cite[Proposition~6]{95Mal} showed that, given an involution~$u$ in $2B$, there is an element~$t$ of order~$3$ such that $L=\langle u,t \rangle$, and therefore  $L=\langle u,u^t, u^{t^2} \rangle$. This assertion combined with Proposition~\ref{TitsGen} suggests that $\alpha(x,L)=2$ if $x$ is an inner automorphism of odd prime order, $\alpha(x,L)=3$ if $x$ is an involution, and $\alpha(x,L)\leqslant 3$ if $x$ is an ``outer'' automorphism  of prime order.

\end{document}